\documentclass[12pt,reqno]{amsart}

\usepackage[T1]{fontenc}
\usepackage[utf8]{inputenc}
\usepackage[english]{babel}

\usepackage{mathrsfs, csquotes, amsfonts,amssymb, amsmath,mathtools,amssymb, enumerate, xcolor}

\theoremstyle{plain}
\newtheorem{theorem}{{Theorem}}[section]
\newtheorem*{theorem*}{{Theorem}}
\newtheorem{proposition}[theorem]{Proposition}

\newtheorem*{proposition*}{Proposition}
\newtheorem{corollary}[theorem]{Corollary}
\newtheorem*{corollary*}{Corollary}
\newtheorem{lemma}[theorem]{Lemma}

\newtheorem{assumption}[theorem]{Assumption}
\newtheorem*{lemma*}{Lemma}
\theoremstyle{definition}
\newtheorem{definition}[theorem]{Definition}
\newtheorem*{definition*}{Definition}
\theoremstyle{remark}
\newtheorem{remark}[theorem]{Remark}
\newtheorem{notation}[theorem]{Notation}

\makeatletter

\@addtoreset{equation}{section}
\makeatother

\usepackage{color}
\definecolor{bleu_sombre}{rgb}{0,0,0.6}
\definecolor{Bl}{rgb}{0,0,0.6}
\definecolor{rouge_sombre}{rgb}{0.8,0,0}
\definecolor{vert_sombre}{rgb}{0,0.6,0}
\usepackage[plainpages=false,colorlinks,linkcolor=bleu_sombre,
citecolor=rouge_sombre,urlcolor=vert_sombre,breaklinks]{hyperref}

\usepackage{pgf,tikz,pgfplots}
\definecolor{webblue}{rgb}{0.22,0.45,0.70}
\definecolor{webred}{rgb}{0.5, 0.09, 0.09}
\definecolor{zzttqq}{rgb}{0.6,0.2,0.}
\usetikzlibrary{arrows}

\renewcommand{\leq}{\leqslant}	
\renewcommand{\geq}{\geqslant}
\newcommand{\C}{\mathbb{C}}
\newcommand{\R}{\mathbb{R}}

\newcommand{\dd}{\mathrm{d}}

\makeatletter

\@addtoreset{equation}{section}
\makeatother

\numberwithin{equation}{section}

\title[]{Spectral asymptotics of the Neumann Laplacian\\ with variable magnetic field on a smooth bounded domain in three dimensions}

\author{M. Aafarani}
\address[M. Aafarani]{Univ Angers, CNRS, LAREMA, SFR MATHSTIC, F-49000 Angers, France}
\email{maha.aafarani94@gmail.com}

\author{K. Abou Alfa}
\address[K. Abou Alfa]{LMJL - UMR6629, Nantes Université, CNRS, 2 rue de la Houssini\`ere, BP 92208, F-44322 Nantes cedex 3, France}
\email{khaled.abou-alfa@univ-nantes.fr}

\author[F. H\'erau]{Fr\'ed\'eric H\'erau}
\address[F. H\'erau]{LMJL - UMR6629, Nantes Université, CNRS, 2 rue de la Houssini\`ere, BP 92208, F-44322 Nantes cedex 3, France}
\email{herau@univ-nantes.fr}

\author[N. Raymond]{Nicolas Raymond}
\address[N. Raymond]{Univ Angers, CNRS, LAREMA, Institut Universitaire de France, SFR MATHSTIC, F-49000 Angers, France}
\email{nicolas.raymond@univ-angers.fr}

\begin{document}
	\maketitle
\begin{abstract}
This article is devoted the semiclassical spectral analysis of the Neumann magnetic Laplacian on a smooth bounded domain in three dimensions. Under a generic assumption on the variable magnetic field (involving a localization of the eigenfunctions near the boundary), we establish a semiclassical expansion of the lowest eigenvalues. In particular, we prove that the eigenvalues become simple in the semiclassical limit.
\end{abstract}
	
	\section{Motivation and main result}
	\subsection{The operator}
Let $\Omega \subset \mathbb{R}^3 $ be a smooth connected open bounded domain. We consider $\textbf{A} : \overline{\Omega} \to \mathbb{R}^3$  a smooth magnetic vector potential. The associated magnetic field is given by \[ \textbf{B}(x) = \nabla \times \textbf{A}(x)\,, \]
and assumed to be non vanishing on $\overline{\Omega}$.
	For $h>0$, we consider the selfadjoint operator 
	\begin{equation}
		\mathscr{L}_h = (-ih \nabla - \textbf{A})^2
	\end{equation}
	with domain 
	\[ \text{Dom}(\mathscr{L}_h) = \{ \psi \in H^2(\Omega) :  \mathbf{n}\cdot(-ih\nabla-\mathbf{A})\psi=0 \mbox{ on } \partial\Omega\}\,,\]
	where $\mathbf{n}$ is the outward pointing normal to the boundary.
	
	The associated quadratic form is defined, for all $\psi\in H^1(\Omega)$, by
	\begin{equation*}
		\forall \psi \in H^1(\Omega)\,,\quad  \mathcal{Q}_h(\psi) = \int_{\Omega} |(-ih \nabla - \textbf{A})\psi|^2 \, \mathrm{d} x.   
	\end{equation*} 
Since $\Omega$ is smooth and bounded, the operator $\mathscr{L}_h$ has compact resolvent and we can consider the non-decreasing sequence of its eigenvalues $(\lambda_n(h))_{n\geq 1}$ (repeated according to their multiplicities). The aim of this article is to describe the behavior of the eigenvalues $\lambda_n(h)$ in the semiclassical limit $h\to 0$.

\subsection{The operator on a half-space with constant magnetic field}
The boundary of $\Omega$ has an important influence on the spectral asymptotics. Let us consider $x_0\in\partial\Omega$ and the angle $\theta(x_0)\in\left[-\frac\pi2,\frac\pi2\right]$ given by 	
\[\mathbf{B}(x_0)\cdot\mathbf{n}(x_0)=\|\mathbf{B}(x_0)\| \sin(\theta(x_0))\,.\]
Near $x_0$, one will approximate $\Omega$ by the half-space $\mathbb{R}^3_+ = \left\{ (r,s,t)\in\mathbb{R}^3 : t>0\right\}$ (the variable $t$ playing the role of the distance to the boundary). Then, this will lead to consider the Neumann realization of
\[\mathfrak{L}_\theta=(D_r-t\cos\theta+s\sin\theta)^2+D_s^2+D_t^2\]
in the ambient space $L^2(\mathbb{R}^3_+)$, which already appeared in \cite{LP00} in the context of Ginzburg-Landau theory. The corresponding magnetic field is $\mathbf{b}(\theta)=(0,\cos\theta,\sin\theta)$. We let
\[\mathbf{e}(\theta)=\inf\mathrm{sp}(\mathfrak{L}_\theta)\,.\]
It is well-known (see \cite{HM02}, \cite{LP00}, and also \cite[Section 2.5.2]{Raymond17}) that $\mathbf{e}$ is even, continuous and increasing on $\left[0,\frac\pi2\right]$ (from $\Theta_0:=\mathbf{e}(0)\in(0,1)$ to $1$) and analytic on $\left(0,\frac\pi2\right)$. Moreover, we can prove that, for all $\theta\in\left(0,\frac\pi2\right)$, $\mathbf{e}(\theta)$ is also the groundstate energy of the Neumann realization of the "Lu-Pan" operator, acting on $L^2(\mathbb{R}^2_+)$,
\begin{equation}\label{eq.LuPan}
\mathcal{L}_\theta=(t\cos\theta-s\sin\theta)^2+D_s^2+D_t^2\,,
\end{equation}
see \cite[Section 0.1.5.4]{Raymond17}.
In this case, the groundstate energy belongs to the discrete spectrum and it is a simple eigenvalue.

These considerations lead to introduce the function $\beta$ on the boundary.
\begin{definition}
We let, for all $x\in\partial\Omega$,
\[\beta(x)=\|\mathbf{B}(x)\|\mathbf{e}(\theta(x))\,.\]	
\end{definition}

\subsection{Context, known results and main theorem}\label{sec.context}
The function $\beta$ plays a central role in the semiclassical spectral asymptotics. The one-term asympotics of $\lambda_1(h)$ is established in \cite{LP00} (see also \cite{R10} and \cite{FH10} where additionnal details are provided).
\begin{theorem}[Lu-Pan '00]\label{thm.LP}
We have
\[\lambda_1(h)=h\min(b_{\min},\beta_{\min})+o(h)\,,\]	
where  $b_{\min}=\min_{x\in\overline{\Omega}}\|\mathbf{B}(x)\|$ and $\beta_{\min}=\min_{x\in\partial\Omega}\beta(x)$.
\end{theorem}
When $\mathbf{B}$ is constant (or with constant norm), more accurate estimates of the groundstate energy have been obtained in \cite{HM04} and in \cite{R10b}. When looking at Theorem \ref{thm.LP}, natural questions can be asked. Can we describe more than the groundstate energy? Is the groundstate energy a simple eigenvalue? In three dimensions, most of the results in this direction have been obtained rather recently: 
\begin{enumerate}[---]
\item When $b_{\min}<\beta_{\min}$, we can prove that the boundary is essentially not seen by the eigenfunctions with low eigenvalues and that they are localized near the minima of $\|\mathbf{B}\|$. Then, if the minimum is unique and non-degenerate, the analysis of \cite{HKRVN16} applies and it can be established that
\[\lambda_n(h)=b_{\min}h+C_0 h^{\frac32}+(C_1(2n-1)+C_2)h^2+o(h^2)\,,\]
where the constants $(C_0, C_1,C_2)\in \mathbb{R}\times\mathbb{R}_+\times\mathbb{R}$ reflect the classical dynamics in a magnetic field.
\item When $\mathbf{B}$ is constant (or with constant norm), we can prove that $\beta_{\min}<b_{\min}$ and that $\beta_{\min}=\Theta_0\|\mathbf{B}\|$. In this case, the eigenfunctions with low eigenvalues are localized near the points of the boundary where the magnetic field is tangent, that is where $\mathbf{e}(\theta(x))$ is minimal. Assuming that the magnetic field becomes generically tangent to the boundary along a nice closed curve and assuming also a non-degeneracy assumption, we have, from \cite{HR22},
\[\lambda_n(h)=\beta_{\min}h+C_0 h^{\frac43}+C_1h^{\frac32}+(C_2(2n-1)+C_3)h^{\frac53}+o(h^2)\,,\]
for some constants $(C_0, C_1,C_2, C_3)\in \mathbb{R}^2\times\mathbb{R}_+\times\mathbb{R}$. The result in \cite{HR22} is stated in the case of a constant magnetic field, but only the fact that its norm is constant is actually used in the analysis, see \cite[Section 3.2.1]{HR22}. Note that without the additionnal non-degeneracy assumption and stopping the analysis before \cite[Section 5.6]{HR22}, this work provides us with the two-term expansion. See also \cite{HK22}.
\end{enumerate}
When $\beta_{\min}<b_{\min}$ and when $\|\mathbf{B}\|$ is variable, it seems that less is known. The first estimates of the low-lying eigenvalues, and not only of the first one, are done in \cite{R10} (see also \cite{Raymond09}), where an upper bound is obtained under a generic assumption (see Assumption \ref{hyp.main} below):
\begin{equation}\label{eq.ub2010}
\lambda_n(h)\leq \beta_{\min}h+C_0 h^{\frac32}+(C_1(2n-1)+C_2)h^2+o(h^2)\,,
\end{equation}
for some constants $(C_0, C_1,C_2)\in \mathbb{R}\times\mathbb{R}_+\times\mathbb{R}$ and where $C_1$ is explicitly given by
\[C_1=\frac{\sqrt{\det\mathrm{Hess}_{x_0}\beta}}{2\|\mathbf{B}(x_0)\|\sin\theta(x_0)}\,.\]
The upper bound \eqref{eq.ub2010} is obtained by means a construction of quasimodes in local coordinates near the minimum of $\beta$ and involves a number of rather subtle algebraic cancellations. At a conference in Dijon in March 2010, S. V\~{u} Ng\d oc suggested to the last author that these algebraic cancellations were the signs of a hidden normal form. At the same conference, J. Sjöstrand also suggested that a dimensonal reduction in the Grushin spirit (see the remarkable survey \cite{SZ07}) could provide us with the lower bound. Retrospectively, we will see that both of them were somewhat right, but that some microlocal technics needed to be developed further in order to tackle the problem in an efficient way. 

Until now, the matching lower bound to \eqref{eq.ub2010} has only been obtained for a toy model in the case of a flat boundary with an explicit polynomial magnetic field, see \cite{R12}. The aim of this article is to establish a lower bound that matches to \eqref{eq.ub2010}. To do so, we will, of course, work under the same assumption as in \cite{R10}. 

\begin{assumption}\label{hyp.main}
	The function $\beta$ has a unique minimum, which is non-degenerate. It is attained at $x_0\in\partial\Omega$ and we have
	\begin{equation}\label{eq.theta}
	\theta(x_0)\in\left(0,\frac\pi2\right)\,.
	\end{equation}
	Moreover, we have
	\[\beta_{\min}=\beta(x_0)=\min_{x\in\partial\Omega}\beta(x)<\min_{x\in\overline{\Omega}} \|\mathbf{B}(x)\|=b_{\min}\,.\]
\end{assumption}

The main result of this article is a three-term expansion of the $n$-th eigenvalue of $\mathscr{L}_h$. Thereby, it completes the picture described above.

\begin{theorem}\label{thm.main}
Under Assumption \ref{hyp.main},  the exist $C_0,C_1\in\mathbb{R}$ such that for all $n\geq 1$, we have
\[\lambda_n(h)\underset{h\to 0}{=}\beta_{\min}h+C_0 h^{\frac32}+\left(\frac{\sqrt{\det\mathrm{Hess}_{x_0}\beta}}{\|\mathbf{B}(x_0)\|\sin\theta(x_0)}\left(n-\frac12\right)+C_1\right)h^2+o(h^2)\,.\]
In particular, for all $n\geq 1$, $\lambda_n(h)$ becomes a simple eigenvalue as soon as $h$ is small enough.	
	\end{theorem}

\subsection{Organization and strategy of the proof}
In Section \ref{sec.2}, we recall the already known results of localization of the eigenfunctions near $x_0$. This formally reduces the spectral analysis to a neighborhood of $x_0$. This suggests to introduce local coordinates near $x_0$. These coordinates $(r,s,t)$ are adapted to the geometry of the magnetic field: the coordinate $s$ is the curvilinear coordinate along the projection of the magnetic field on the boundary (we use here that $\theta(x_0)<\frac\pi2$), the coordinate $r$ is the geodesic coordinate transverse to $s$, and $t$ is the distance to the boundary. A rather similar coordinate system has been used and described in \cite{HR22} (inspired from \cite{HM04}). Then, the local action of the operator is described in Section \ref{sec.Laplacianrst} where we perform a Taylor expansion with respect to the normal variable $t$ only. After a local change of gauge, this makes an approximate magnetic vector potential appear, see \eqref{eq.tildeA3}. In Section \ref{sec.ext}, we define a new operator on $L^2(\mathbb{R}^
 3_+)$ by extending the coefficients, seen as functions of $(r,s)$ defined near $(0,0)$, to functions on $\mathbb{R}^2$. Since this extension occurs away from the localization zone of the eigenfunctions, we get a new operator $\mathscr{L}_h^{\mathrm{app}}$ whose spectrum is close to that of $\mathscr{L}_h$, see Proposition \ref{prop.Lapp}.
 
In Section \ref{sec.3}, we perform the analysis of $\mathscr{L}_h^{\mathrm{app}}$ with the help of the change of coordinates $(r,s)\mapsto\mathscr{J}(r,s)=(u_1,u_2)$, whose geometric role is to make the normal component of the magnetic field constant (here, we use $\theta(x_0)>0$). This idea is reminiscent of the recent work \cite{MRVN22} in two dimensions, see \cite[Prop. 2.2]{MRVN22}. We are reduced to the spectral analysis of the operator $\mathscr{N}_h$, see \eqref{eq.LapptoNh}. Then, we conjugate $\mathscr{N}_h$ by means a tangential Fourier transform (in the direction $u_1$) and a translation/dilation $T$ (after these transforms, the variable $u_1$ becomes $z$). After these explicit transforms, we get a new operator $\mathscr{N}^\sharp_\hbar$, which can be seen as a differential operator of order two in the variables $(z,t)$ with coefficients that are $h$-pseudodifferential operators (with an expansion in powers of $\hbar=h^{\frac12}$) in the variable $u_2$ only, see \eqref{eq.Ndiese}. Its eigenfunctions are localized in $(z,t)$, see Proposition \ref{prop.variable-effective2} and Remark \ref{rem.roughfrequency}.

In Section \ref{sec.4}, this localization with respect to $z$ suggests to insert cutoff functions in the coefficients of our operator. By doing this, we get the operator $\mathscr{N}_\hbar^\flat$, see \eqref{eq.Nbemol}. The advantage of $\mathscr{N}_\hbar^\flat$ is that it can be considered as a pseudodifferential operator with operator-valued symbol in a reasonable class $S(\mathbb{R}^2,N)$, see Proposition \ref{prop.Nhbarbemol}. The principal operator symbol $n_0(u,\upsilon)$ is unitarily equivalent to the Lu-Pan operator $\|\mathbf{B}(\upsilon,-u)\|\mathcal{L}_{\theta(\upsilon,-u)}$ (where we make here a slight abuse of notation by forgetting the reference to the local coordinates on the boundary), see Proposition \ref{prop.psymb}. Then, we may construct an inverse for $n_0-\Lambda$ by means of the so-called Grushin formalism, as soon as $\Lambda$ is close to $\beta_{\min}$, see Lemma \ref{lem.Grushin0}. This is the first step in the approximate parametrix construction for $\mathscr{N}^\flat_\hbar-\Lambda$ given in Proposition \ref{prop.parametrix}, which is the key of the proof of Theorem \ref{thm.main}. Let us emphasize that this parametrix construction is inspired by \cite{Keraval} and based on ideas developed by A. Martinez and J. Sjöstrand. This formalism has recently been used in \cite{HR22} in three dimensions (see also \cite{BHR21, FHKR22, FLTRVN22} in the case of two dimensions). At a formal level, this parametrix construction relates the kernel of $\mathscr{N}_\hbar^\flat-\Lambda$ to that of an effective pseudodifferential operator $Q^\pm_\hbar(\Lambda)$, see \eqref{eq.Qpm}. 

Section \ref{sec.5} is devoted to relate the spectrum of $\mathscr{N}_\hbar^\sharp$ to that of the effective operator $(p^{\mathrm{eff}}_\hbar)^W$, see \eqref{eq.peff}. Note that the effective operator is an operator in one dimension. This contrasts with \cite{HR22} where a double Grushin reduction is used: here, this reduction is done in one step with the help of the Lu-Pan operator. The quasi-parametrix in Proposition \ref{prop.parametrix} is the bridge between the spectra of $\mathscr{N}_\hbar^\sharp$ and $(p^{\mathrm{eff}}_\hbar)^W$. 

We emphasize that we have to be very careful when studying 
 this connection since the symbol of the effective operator is not necessarily real-valued (only its principal symbol $p_0$ is a priori real). This contrasts again with \cite{HR22} and all the previous works on the subject. This non-selfadjointness comes from the fact that $\mathscr{N}_h$ is not selfadjoint on the canonical $L^2$-space, but on a weighted $L^2$-space. That is why a short detour in the world of non-selfadjoint operators is used in Section \ref{sec.5}. In fact, one will not need the operator $(p_\hbar^{\mathrm{eff}})^W$ more than its approximation $(p_\hbar^{\mathrm{mod}})^W$ near the minimum of $p_0$, see Section \ref{sec.model}. This approximation is a complex perturbation of the harmonic oscillator. Its spectrum is well-known as well as the behavior of its resolvent.
 
In Section \ref{sec.quasimodes}, we use rescaled Hermite functions to construct quasimodes for $\mathscr{N}_\hbar^\sharp$. This shows that the spectrum of the model operator is in fact real and we get an accurate upper bound of $\lambda_n(\mathscr{N}_\hbar^\sharp)$ in \eqref{eq.upperbound}. This reproves in a much shorter way \eqref{eq.ub2010} (see \cite[Theorem 1.5]{R10} where the convention $\|\mathbf{B}(x_0)\|=1$ is used). Section \ref{sec.lowerbound} is devoted to establish the corresponding lower bound (by using in particular that the eigenvalues of the non-selfadjoint operator $(p^{\mathrm{mod}}_\hbar)^W$ have algebraic multiplicity $1$).

\section{Localization near $x_0$ and consequences}\label{sec.2}

\subsection{Localization estimates}
In this section, we gather some already known localization properties of the eigenfunctions, see \cite{Raymond09}.

\begin{proposition}[Localization near the boundary]\label{prop.loct}
	Under Assumption \ref{hyp.main}, for all $\epsilon>0$ such that $\beta_{\min}+\epsilon<b_{\min}$, there exist $\alpha, C, h_0>0$ such that, for all $h\in (0, h_0)$ and all eigenfunctions $\psi$ of $\mathscr{L}_h$ associated with an eigenvalue $\lambda \leq (\beta_{\min} + \epsilon) h $, we have
	\begin{equation}\label{Agm-1}
		\int_{\Omega} e^{  \frac{2\alpha\mathrm{dist}(x,\partial \Omega)}{\sqrt{h}} } | \psi|^2 \mathrm{d} x \leq C\|\psi\|^2.
	\end{equation}
\end{proposition}
For $\delta>0$, we consider the $\delta$-neighborhood of the boundary given by
\[\Omega_{\delta}:=\left\{ x\in\Omega:\operatorname{dist}(x,\partial\Omega)<\delta \right\}\,.\]
Due to Proposition \ref{prop.loct}, in the following, we take \[\delta=h^{\frac{1}{2}-\eta}\] 
for $\eta\in(0,\frac{1}{2})$. We consider $\mathscr{L}_{h,\delta}=\left( -ih\nabla-\mathbf{A} \right)^2$ the operator with
magnetic Neumann condition on $\partial\Omega$ and Dirichlet condition on $\partial\Omega_{\delta}\setminus\partial\Omega$.
\begin{corollary}\label{cor.loct}
Let $n\geq 1$. There exist $C, h_0>0$ such that for all $h\in(0,h_0)$,
\[ \lambda_n(\mathscr{L}_{h,\delta})-Ce^{-Ch^{-\eta}}\leq\lambda_n(\mathscr{L}_h)\leq \lambda_n(\mathscr{L}_{h,\delta})\,.\]	
\end{corollary}
Thanks to Corollary \ref{cor.loct}, we may focus on the spectral analysis of $\mathscr{L}_{h,\delta}$. The following proposition can be found in \cite[Chapter 9]{FH10} and \cite[Theorem 4.3]{HM02} (see also the proof of \cite[Prop. 2.9]{HR22}).

\begin{proposition}[Localization near $x_0$]\label{prop.locx0}
	    Let $M>0$. There exist $C, h_0>0$ and $\alpha>0$ such that, for all $h\in(0,h_0)$, and all eigenfunctions $\psi$ of $\mathscr{L}_{h,\delta}$ associated with an eigenvalue $\lambda$ such that $\lambda\leq\beta_{\min}h+Mh^{\frac{3}{2}}$, we have
	\begin{equation}
	\int_{\Omega_{\delta}}e^{\frac{2\alpha\mathrm{dist}(x,\partial\Omega)}{\sqrt{h}}}\left| \psi(x) \right|^2\mathrm{d}x+\int_{\Omega_{\delta}}e^{\frac{2\alpha\|x-x_0\|^2}{h^{1/4}}}\left| \psi(x) \right|^2\mathrm{d}x\leq C\left\|\psi\right\|^2\,.
		\label{I1}
	\end{equation}
\end{proposition}
Proposition \ref{prop.locx0} invites us to consider a local chart near $x_0$ and to write the operator in the corresponding coordinates. In order to simplify our analysis, we construct below a system of coordinates compatible with the geometry of the magnetic field.

\subsection{Adapted coordinates near $x_0$}
This section is devoted to introduce coordinates adapted to the magnetic field. Most of the properties of our coordinates system have been established in \cite{HR22}.
\subsubsection{Coordinate in the direction of the magnetic field on the boundary}
We set
\[\mathbf{b}(x)=\frac{\mathbf{B}(x)}{\|\mathbf{B}(x)\|}\,,\]
and we consider its projection on the tangent plane at $x\in\partial\Omega$:
\[\mathbf{b}^{\parallel}(x)=\mathbf{b}(x)-\langle\mathbf{b}(x),\mathbf{n}(x)\rangle\mathbf{n}(x)\,,\]
where $\mathbf{n}$ is the outward pointing normal.

Due to Assumption \ref{hyp.main}, near $x_0$, the vector field $\mathbf{b}^{\parallel}$ does not vanish. This allows to consider the unit vector field
\[\mathbf{f}(x)=\frac{\mathbf{b}^{\parallel}(x)}{\|\mathbf{b}^{\parallel}(x)\|}\]
and the associated integral curve $\gamma$ given by
\[\gamma'(s)=\mathbf{f}(\gamma(s))\,,\quad \gamma(0)=x_0\,,\]
which is well-defined on $(-s_0,s_0)$ for some $s_0>0$.  Clearly, $\gamma$ is smooth and with values in $\partial\Omega$.

\subsubsection{Coordinates on the boundary}
Denoting by $K$ the second fundamental form of $\partial\Omega$ associated to the Weingarten map defined by
\[\forall U,V\in T_x\partial\Omega\,,\quad K_x(U,V)=\langle \mathrm{d} \mathbf{n}_x(U),V\rangle\,,\]
we can consider the ODE with parameter $s$ of unknown $r \mapsto \gamma(r,s)$
\[\partial^2_r\gamma(r,s)=-K(\partial_r\gamma(r,s), \partial_r\gamma(r,s))\mathbf{n}(\gamma(r,s))\,,\]
with initial conditions
\[\gamma(0,s)=\gamma(s)\,,\quad \partial_r\gamma(0,s)=-\gamma'(s)^{\perp}\,,\]
where $\perp$ is understood in the tangent space. The minus is here so that $(\partial_r\gamma,\partial_s\gamma,\mathbf{n})$ is a \emph{direct} orthonormal basis. This ODE has a unique smooth solution $(-r_0,r_0)\times(-s_0,s_0)\ni(r,s)\mapsto\gamma(r,s)$ where $r_0>0$ is chosen small enough. Let us gather the important properties of $(r,s)\mapsto\gamma(r,s)$. Their proofs may be found in \cite{HR22}.

\begin{proposition}\label{prop.grs}
The function $(r,s)\mapsto\gamma(r,s)$ is valued in $\partial\Omega$.
Moreover, we have 
\[|\partial_r\gamma(r,s)|=1\,,\quad\langle\partial_r\gamma,\partial_s\gamma\rangle=0\,.\]
In this chart $\gamma$, the first fundamental form on $\partial\Omega$ is given by the matrix
\[g(r,s)=\begin{pmatrix}
			1&0\\
			0&\alpha(r,s)
\end{pmatrix}
		\,,\quad \alpha(r,s)=|\partial_s\gamma(r,s)|^2\,.\]	
For all $s \in (-s_0,s_0)$, we have $\alpha(0,s)=1$ and $\partial_s\alpha(0,s)=0$.
	
\end{proposition}

\subsubsection{Coordinates near the boundary}
We consider the tubular coordinates associated with the chart $\gamma$:
\begin{equation}\label{eq.Gamma}
y=(r,s,t)\mapsto \Gamma(r,s,t) =\gamma(r,s)-t\mathbf{n}(\gamma(r,s))=x\,.
\end{equation}
The map $\Gamma$ is a smooth diffeomorphism from $Q_0:=(-r_0,r_0)\times(-s_0, s_0)\times(0,t_0)$ to $\Gamma(Q_0)$, as soon as $t_0>0$ is chosen small enough.
The differential of $\Gamma$ can be written as
\begin{equation}\label{eq.dPhiy}
	\mathrm{d}\Gamma_{y}=[(\mathrm{Id}-t\mathrm{d}\mathbf{n})(\partial_r\gamma),(\mathrm{Id}-t\mathrm{d}\mathbf{n})(\partial_s\gamma),-\mathbf{n}]\,,
\end{equation}
and the Euclidean metrics becomes
\begin{equation}\label{eq.G}
	\mathbf{G}=(\mathrm{d}\Gamma)^{\mathrm{T}}\mathrm{d}\Gamma=
	\begin{pmatrix}
		\mathbf{g}&0\\
		0&1
	\end{pmatrix}
	\,,
\end{equation}
with
\[
\mathbf{g}(r,s,t)=
\begin{pmatrix}
	\|(\mathrm{Id}-t\mathrm{d}\mathbf{n})(\partial_r\gamma)\|^2&\langle(\mathrm{Id}-t\mathrm{d}\mathbf{n})(\partial_r\gamma),(\mathrm{Id}-t\mathrm{d}\mathbf{n})(\partial_s\gamma)\rangle\\
	\langle(\mathrm{Id}-t\mathrm{d}\mathbf{n})(\partial_r\gamma),(\mathrm{Id}-t\mathrm{d}\mathbf{n})(\partial_s\gamma)\rangle&\|(\mathrm{Id}-t\mathrm{d}\mathbf{n})(\partial_s\gamma)\|^2
\end{pmatrix}
\,.\]
We have $g(r,s)=\mathbf{g}(r,s,0)$, where $g$ is defined in Proposition \ref{prop.grs}.
\subsubsection{The magnetic form in tubular coordinates}
In this section, we discuss the expression of the magnetic field in the coodinates induced by $\Gamma$. This discusssion can be found in \cite[Section 0.1.2.2]{Raymond17} and \cite[Section 3.2]{HR22}. We consider the $1$-form
\[\sigma=\mathbf{A}\cdot\mathrm{d} x=\sum_{\ell=1}^3 A_\ell\mathrm{d} x_{\ell}\,.\]
Its exterior derivative is the magnetic $2$-form
\[\omega=\mathrm{d} \sigma=\sum_{1\leq k<\ell\leq 3}(\partial_k A_{\ell}-\partial_{\ell}A_k)\mathrm{d} x_k\wedge\mathrm{d} x_{\ell}\,,\]
which can also be written as
\[\omega=B_3\mathrm{d} x_1\wedge \mathrm{d} x_2-B_2\mathrm{d} x_1\wedge \mathrm{d} x_3+B_1\mathrm{d} x_2\wedge \mathrm{d} x_3\,.\]
Note also that
\[\forall U,V\in\mathbb{R}^3\,,\quad \omega(U,V)=[U,V,\mathbf{B}]=\langle U\times V,\mathbf{B}\rangle\,.\]
Let us now consider the effect of the change of variables $\Gamma(y) = x$. We have
\begin{equation}\label{eq.tildeA}
	\Gamma^*\sigma=\sum_{j=1}^3 \tilde A_j \mathrm{d} y_j\,,\quad \tilde{\mathbf{A}}=(\mathrm{d}\Gamma)^{\mathrm T}\circ\mathbf{A}\circ\Gamma\,,
\end{equation}
and
\[\Gamma^*\omega=\Gamma^*\mathrm{d}\sigma=\mathrm{d} (\Gamma^*\sigma)=[\cdot,\cdot,\nabla\times\tilde{\mathbf{A}}]\,.\]
This also gives that, for all $U,V\in\mathbb{R}^3$,
\[[\mathrm{d}\Gamma(U),\mathrm{d}\Gamma(V),\mathbf{B}]=[U,V,\nabla\times\tilde{\mathbf{A}}]\,,\quad
\mbox{ or } \det\mathrm{d}\Gamma[\cdot,\cdot,\mathrm{d}\Gamma^{-1}(\mathbf{B})]=[\cdot,\cdot,\nabla\times\tilde{\mathbf{A}}]\,,\]
so that,
\[\nabla\times\tilde{\mathbf{A}}=(\det\mathrm{d}\Gamma)\, \mathrm{d}\Gamma^{-1}(\mathbf{B})\,.\]
Note then that using \eqref{eq.G} we get
\begin{equation}\label{eq.coordinatesnewB}
	|\mathbf{g}|^{-\frac12}\nabla\times\tilde{\mathbf{A}}=\mathcal{B}\,,
\end{equation}
where $\mathcal{B}(y):=\mathrm{d}\Gamma^{-1}_y(\mathbf{B}(x))$ corresponds to the coordinates of $\mathbf{B}(y)$ in the image of the canonical basis by $\mathrm{d}\Gamma_y$. With our specific change of coordinates \eqref{eq.Gamma}, we have
\[\mathbf{B}=\mathrm{d}\Gamma(\mathcal{B})=\mathcal{B}_1\left( \operatorname{Id}-t\dd\mathbf{n} \right)\left( \partial_r\gamma \right)+\mathcal{B}_2\left( \operatorname{Id}-t\dd\mathbf{n} \right)\left( \partial_s\gamma \right)-\mathcal{B}_3\mathbf{n}\,.\]
For all $x\in\partial \Omega$, \emph{i.e.} $t=0$, we have
\begin{equation}\label{eq.newB}
\begin{split}\mathbf{B}(x)&=\mathcal{B}_1(r,s,0)\partial_r\gamma+\mathcal{B}_2(r,s,0) \partial_s\gamma -\mathcal{B}_3(r,s,0)\mathbf{n}(\gamma(r,s))\,,\\
\left\| \mathbf{B}(x) \right\|^2&=\mathcal{B}_1^2(r,s,0)+\alpha(r,s)\mathcal{B}_2^2(r,s,0)+\mathcal{B}_3^2(r,s,0)\,.\end{split}
\end{equation}
Moreover, we have
\[\mathcal{B}_1(r,s,0)=\langle \mathbf{B},\partial_r\gamma \rangle\,,\qquad\alpha(r,s)\mathcal{B}_2(r,s,0)=\langle \mathbf{B},\partial_s\gamma \rangle\,,\qquad\mathcal{B}_3(r,s,0)=-\langle \mathbf{B},\mathbf{n} \rangle\,.\]
Note that our choice of coordinate $s$ (along the projection of the magnetic field on the tangent plane) and of transverse coordinate $r$ implies that
\[\mathcal{B}_1(0,s,0)=0\,,\quad \mathcal{B}_2(0,s,0)>0\,,\]
thanks to Assumption \ref{hyp.main}.
\begin{definition}
In a neighbordhood of $(0,0)$, we can consider the unique smooth function $\theta$ such that
\[\mathbf{B}\left( \gamma(r,s) \right)\cdot\mathbf{n}\left( \gamma(r,s) \right)=\left\| \mathbf{B}\left( \gamma(r,s) \right) \right\|\sin\theta(r,s)\]
and satisfying $\theta(r,s)\in\left(0,\frac\pi2\right)$. With a sligh abuse of notation, we let
\[\beta(r,s)=\|\mathbf{B}(\gamma(r,s))\|\mathbf{e}(\theta(r,s))\,.\]
\end{definition}
\begin{remark}\label{rem.B3}
We have
\[\mathcal{B}_3(r,s)=-\left\| \mathbf{B}(\gamma(r,s)) \right\|\sin\left( \theta(r,s) \right)\,.\]
Moreover, since $\mathcal{B}_2>0$ and $\alpha(0,s)=1$,
\[\mathcal{B}_2(0,s,0)=\|\mathbf{B}(\gamma(0,s))\|\cos\theta(0,s)\,,\quad \mathcal{B}_3(0,s,0)=-\|\mathbf{B}(\gamma(0,s))\|\sin\theta(0,s)\,.\]
\end{remark}
In fact, we can choose a suitable explicit $\tilde{\mathbf{A}}$ such that \eqref{eq.coordinatesnewB} holds in a neighborhood of $(0,0,0)$.
\begin{lemma}\label{lem.changegauge}
Considering
\[\begin{split}
	\tilde A_1(r,s,t)&=\int_0^t[|\mathbf{g}|^{\frac12}\mathcal{B}_2](r,s,\tau)\mathrm{d}\tau\,,\\
	\tilde A_2(r,s,t)&=  -\int_0^t[|\mathbf{g}|^{\frac12}\mathcal{B}_1](r,s,\tau)\mathrm{d}\tau+\int_0^r [|\mathbf{g}|^{\frac12}\mathcal{B}_3](u,s,0) \mathrm{d}u\,,\\
	\tilde A_3(r,s,t)&=0\,,
\end{split}\]	
we have $\nabla\times\tilde{\mathbf{A}}(r,s,t)=|\mathbf{g}|^{\frac12}\mathcal{B}(r,s,t)$.
\end{lemma}
\begin{proof}
It follows from a straightforward computation and the fact that $|\mathbf{g}|^{\frac12}\mathcal{B}$ is divergence-free.
\end{proof}

\begin{remark}
Note that the proof of Lemma \ref{lem.changegauge} does not involve global geometric quantities on the boundary as in \cite[Prop. 3.3]{HR22} since our analysis is local near $x_0$.	
\end{remark}

\subsection{First approximation of the magnetic Laplacian in local coordinates}\label{sec.Laplacianrst}
If the support of $\psi$ is close enough to $x_0$, we may express $\mathcal{Q}_h(\psi)$ in the local chart given by $\Gamma(y) = x$. Letting $\tilde\psi(y)=\psi\circ\Gamma(y)$, we have then
\[\mathcal{Q}_h(\psi)=\int \langle \mathbf{G}^{-1}(-ih\nabla_y-\tilde{\mathbf{A}}(y))\tilde\psi,(-ih\nabla_y-\tilde{\mathbf{A}}(y))\tilde\psi \rangle |\mathbf{g}|^{\frac12}\mathrm{d} y\,.\]
In the Hilbert space $L^2(|\mathbf{g}|^{\frac12}\mathrm{d} y)$, the operator locally takes the form
\begin{equation}\label{eq.magLaptilde}
	|\mathbf{g}|^{-\frac12}(-ih\nabla_y-\tilde{\mathbf{A}}(y))\cdot|\mathbf{g}|^{\frac12}\mathbf{G}^{-1}(-ih\nabla_y-\tilde{\mathbf{A}}(y))\,,
\end{equation}
where $\mathbf{G}$ is defined in \eqref{eq.G}. From now on, the analysis deviates from \cite{HR22}.

\subsubsection{Expansion with respect to $t$}
Due to the localization near the boundary at the scale $h^{\frac12}$, we are led to replace $\tilde{\mathbf{A}}$ by its Taylor expansion $\tilde{\mathbf{A}}^{[3]}$ at order $3$ and $\mathbf{g}$ and $\mathbf{G}$ by their Taylor expansions at the order $2$. We let
\begin{equation}\label{eq.tildeA3}
\begin{split}
	\tilde A_1^{[3]}(r,s,t)&=t[|\mathbf{g}|^{\frac12}\mathcal{B}_2](r,s,0)+C_2\hat t^2+C_3\hat t^3\,,\\
	\tilde A_2^{[3]}(r,s,t)&=  -t[|\mathbf{g}|^{\frac12}\mathcal{B}_1](r,s,0)+F(r,s)+E_2\hat t^2+E_3\hat t^3\,,\\
	\tilde A^{[3]}_3(r,s,t)&=0\,,
\end{split}
\end{equation}
where $\hat t=t\chi(h^{-\frac12+\eta}t)$ for some smooth cutoff function $\chi$ equal to $1$ near $0$ and where
\begin{equation}\label{eq.F}
F(r,s)=\int_0^r [|\mathbf{g}|^{\frac12}\mathcal{B}_3](\ell,s,0) \mathrm{d}\ell\,,
\end{equation}
and the functions $C_j(r,s)$ and $E_j(r,s)$ are smooth. We emphasize that we only truncate the terms of order at least $2$ in $t$ in the above expression.
 
Due to Assumption \ref{hyp.main}, $(r,s)\mapsto (F(r,s),s)$ is a smooth diffeomorphism on a neighborhood of $(0,0)$.

We also consider the expansions
\[|\mathbf{g}|^{\frac12}(r,s,t)=m(r,s,t)+\mathscr{O}(t^3)\,,\quad \mathbf{G}^{-1}=M(r,s,t)^{-1}+\mathscr{O}(t^3)\,,\]
with
\[m(r,s,t)=a_0(r,s)+\hat ta_1(r,s)+\hat t^2a_2(r,s)\,,\quad M(r,s,t)=M_0(r,s)+\hat tM_1(r,s)+\hat t^2M_2(r,s)\,.\]
Recall that $|\mathbf{g}|(r,s,0)=\alpha(r,s)$.

\subsubsection{Extension of the functions of the tangential variables}\label{sec.ext}

It will be convenient to work on the half-space $\mathbb{R}^3_+$ instead of a neighborhood of $(0,0,0)$. 


Given $\epsilon_0>0$, consider a smooth odd function $\zeta : \mathbb{R}\to\mathbb{R}$ such that $\zeta(x)=x$ on $[0,\epsilon_0]$ and $\zeta(x)=2\epsilon_0$, for all $x\geq 2\epsilon_0$. In particular, $\|\zeta\|_\infty=2\epsilon_0$. We let
\[Z(r,s)=(\zeta(r),\zeta(s))\,.\]

The following lemma is a straightforward consequence of Assumption \ref{hyp.main}.
\begin{lemma}\label{lem.epsilon0}
For $\epsilon_0$ small enough, the function $\hat\beta=\beta\circ Z : \mathbb{R}^2\to\mathbb{R}_+$ is smooth and has a unique minimum (at $(0,0)$), which is non-degenerate and not attained at infinity.
\end{lemma}

Let us now replace the function $\mathscr{B} : (r,s)\mapsto \alpha(r,s)^{\frac12}\mathcal{B}(r,s,0)$ by $\mathscr{B}\circ Z$ in \eqref{eq.tildeA3} and \eqref{eq.F}. We replace the other coefficients $C_j$ and $E_j$ by $C_j\circ Z$ and $E_j\circ Z$. Note that we have the following.

\begin{lemma}\label{lem.J}
For $\epsilon_0$ small enough, the function 
\[\mathscr{J} : \mathbb{R}^2\ni(r,s)\mapsto \left(\int_0^r [|\mathbf{g}|^{\frac12}\mathcal{B}_3](Z(\ell,s),0) \mathrm{d}\ell,s\right)=u=(u_1,u_2)\in\mathbb{R}^2\] is smooth and it is a global diffeomorphism.
\end{lemma}
This leads to consider the new vector potential
\begin{equation}\label{eq.hatA}
	\begin{split}
		 \hat A_1(r,s,t)&=t\overset{\circ}{C}_1+\overset{\circ}{C}_2\hat t^2+\overset{\circ}{C}_3\hat t^3\,,\\
		\hat A_2(r,s,t)&=  -t\overset{\circ}{E}_1+\mathscr{J}_1(r,s)+\overset{\circ}{E}_2\hat t^2+\overset{\circ}{E}_3\hat t^3\,,\\
		\hat A_3(r,s,t)&=0\,,
	\end{split}
\end{equation}
where $C_1=\alpha^{\frac12}\mathcal{B}_2$, $E_1=\alpha^{\frac12}\mathcal{B}_1$ and with the notation $\overset{\circ}{f}=f\circ Z$.

The rest of the article will be devoted to the spectral analysis of the operator associated with the new quadratic form
\[\mathcal{Q}^{\mathrm{app}}_h(\varphi)=\int_{\mathbb{R}^3_+}\langle (\overset{\circ}{M})^{-1}(-ih\nabla_y-\hat{\mathbf{A}}(y))\varphi,(-ih\nabla_y-\hat{\mathbf{A}}(y))\varphi \rangle \overset{\circ}{m}\mathrm{d} y\,.\]
This selfadjoint operator $\mathscr{L}^{\mathrm{app}}_h$ is acting as
\[\overset{\circ}{m}^{-1}(-ih\nabla_y-\hat{\mathbf{A}})\cdot\overset{\circ}{m}(\overset{\circ}{M})^{-1}(-ih\nabla_y-\hat{\mathbf{A}})\,,\]
in the ambient Hilbert space $L^2(\mathbb{R}^3_+,\overset{\circ}{m}\mathrm{d}y)$. This spectral analysis is motivated by the fact that the low-lying spectra of $\mathscr{L}_h$ and $\mathscr{L}^{\mathrm{app}}_h$ coincide modulo $o(h^2)$, in the sense of the following proposition.
\begin{proposition}\label{prop.Lapp}
We have, for all $n\geq 1$,
\[\lambda_n(h)=\lambda_n(\mathscr{L}^{\mathrm{app}}_h)+o(h^2)\,.\]	
\end{proposition}
We omit the proof. It follows from Corollary \ref{cor.loct} and the localization estimates given in Proposition \ref{prop.locx0} (which are also true in the coordinates $(r,s,t)$ for those of $\mathscr{L}^{\mathrm{app}}_h$ by using the same arguments) and the Min-max Theorem. These localisation estimates allow to remove the cutoff functions up to remainders of order $\mathscr{O}(h^\infty)$ and to control the remainders of the expansion in $t$.

\section{Change of coordinates and metaplectic transform}\label{sec.3}
In order to perform the spectral analysis of $\mathscr{L}^{\mathrm{app}}_h$, it is convenient to use the change of variable $\mathscr{J}$ given in Lemma \ref{lem.J}. More precisely, we will use the unitary transform induced by $\mathscr{J}$ defined by
\[U:\begin{array}{ccc}	
L^2(\mathbb{R}^3_+,\overset{\circ}{m}\dd y)	&\to& L^2(\mathbb{R}^3_+,\breve m\left|\mathrm{Jac}\mathscr{J}^{-1}\right|\dd u\dd t)\\
\varphi	&\mapsto&\breve\varphi\\
	\end{array}\,,\]
where we used the notation $\breve f(u,t)=f(\mathscr{J}^{-1}(u),t)$ and the slight abuse of notation $\breve{\overset{\circ}{f}}=\breve f$. Then, we focus on the operator $\mathscr{N}_h=U\mathscr{L}^{\mathrm{app}}_hU^{-1}$, acting in $L^2(\mathbb{R}^3_+,\breve m\left|\mathrm{Jac}\mathscr{J}^{-1}\right|\dd u\dd t)$. The operator $\mathscr{N}_h$ is acting as 
\begin{equation}\label{eq.LapptoNh}
\mathscr{N}_h=U\mathscr{L}^{\mathrm{app}}_hU^{-1}={\breve m}^{-1}\mathscr{D}_h\cdot{\breve m}({\breve M})^{-1}\mathscr{D}_h\,,
\end{equation}
where
\[\mathscr{D}_h=\begin{pmatrix}
	-ih{\breve C}_0\partial_{u_1}-t{\breve C}_1-\hat t^2{\breve C}_2-\hat t^3{\breve C}_3	\\
	-ih\partial_{u_2}-u_1-ih\breve E_0\partial_{u_1}+t{\breve E}_1-\hat t^2{\breve E}_2-\hat t^3{\breve E}_3	\\
	-ih\partial_t
\end{pmatrix}\,,\]
and
\begin{equation}\label{eq.C0D0}
C_0=\partial_r\mathscr{J}_1=\alpha^{\frac12}\mathcal{B}_3\,,\quad   E_0=\partial_s\mathscr{J}_1\,.
\end{equation}

\begin{notation}
We will use the following classical notation for the semiclassical Weyl quantization of a symbol $a=a(u,\upsilon)$. We let
\[a^W\psi(u)=\frac{1}{(2\pi h)^2}\int_{\mathbb{R}^4}e^{i(u-x)\cdot\upsilon/h}a\left(\frac{u+x}{2},\upsilon\right)\psi(x)\dd x\dd\upsilon\,.\]
\end{notation}

\begin{proposition}\label{prop.variable-effective}
Let $K>0$ and $\eta\in\left(0,\frac12\right)$. Let $\Xi$ be a smooth function equal to $0$ near $0$ and $1$ away from a compact neighborhood of $0$. There exists $h_0>0$ such that for all $h\in(0,h_0)$ and for all normalized eigenfunctions $\psi$ of $\mathscr{N}_h$ associated with an eigenvalue $\lambda$ such that $\lambda\leq Kh$, we have
\[\left[\Xi\left(\frac{u_1-\upsilon_2}{h^{\frac12-\eta}}\right)\right]^W\psi=\mathscr{O}(h^\infty)\,.\]
\end{proposition}
\begin{proof}
To simplify the notation, we denote by $\Xi_h=\Xi\left( \frac{u_1-\upsilon_2}{h^{\frac{1}{2}-\eta}} \right)$. Let $\psi$ be a normalized eigenfunction of $\mathscr{N}_h$ associated with an eigenvalue $\lambda$ such that $\lambda\leq Kh$. The eigenvalue equation gives us
\begin{equation}
\langle \mathscr{N}_h \Xi_h^W\psi,\Xi_h^W\psi \rangle=\lambda \| \Xi_h^W\psi \|^2+\langle \left[ \mathscr{N}_h,\Xi_h^W \right]\psi,\Xi_h^W\psi \rangle,
    \label{E1}
\end{equation}
where $\langle\cdot,\cdot\rangle$ is the scalar product in $L^2\left( \mathbb{R}_+^3,\breve m | \mathrm{Jac}\mathscr{J}^{-1} |\dd u \dd t \right)$.

According to the localization at the scale $h^{\frac{1}{2}}$ with respect to $t$, we can insert a cutoff function supported in $\{t\leq h^{\frac{1-\eta}{2}} \}$ and we obtain, for $j=2,3$,
\begin{equation}
\| t^j\Xi_h^W\psi  \|\leq Ch^{1-\eta}\| \Xi_h^W\psi  \|+\mathscr{O}(h^{\infty})\| \psi \|\,.
 \label{LOC}
\end{equation}
By means of the Young inequality and rough quadratic form estimates, this yields, for some $c, C>0$,
\begin{equation}
    \langle \mathscr{N}_h \Xi_h^W\psi,\Xi_h^W\psi \rangle\geq cQ_h^0(\Xi_h^W\psi)-Ch^{1-\eta}\| \Xi_h^W\psi\|^2+\mathscr{O}(h^{\infty})\|\psi\|^2\,,
    \label{LOC1}
\end{equation}
where 
\[
Q_h^0(\varphi)=\|h\partial_t\varphi\|^2+\left\| ( h\breve C _0D_{u_1}-t\breve C _1)\varphi \right\|^2+\left\| (hD_{u_2}-u_1+h\breve E_0 D_{u_1}+t\breve E _1)\varphi \right\|^2\,.
\]
Then, using again the Young inequality, we find that
\[Q_h^0(\varphi)\geq \|h\partial_t\varphi\|^2+\frac12\left\|h\breve C _0D_{u_1}\varphi \right\|^2+\frac12\left\| (hD_{u_2}-u_1)\varphi \right\|^2-2\|h\breve E _0 D_{u_1}\varphi\|^2-C\|t\varphi\|^2\,.\]
Notice that there exists $c>0$ such that 
\[
| \breve C _0  | \geq c\,,\quad  | \breve E_0 | \leq \frac c4 \,,
\]
where we recall \eqref{eq.C0D0} and Lemma \ref{lem.J}. Note also that $C_0$ is globally positive and that $E_0$ is as small as we want since it vanishes at $(0,0,0)$, after the extension procedure in Section \ref{sec.ext}. This shows that, for some $c_0>0$,
\begin{equation}\label{eq.Q0h}
Q_h^0(\varphi)\geq \|h\partial_t\varphi\|^2+c_0\left\|hD_{u_1}\varphi \right\|^2+\frac12\left\| (hD_{u_2}-u_1)\varphi \right\|^2-C\|t\varphi\|^2\,.
\end{equation}
On the support of $\Xi_h$ we have $(\upsilon_2-u_1)^2\geq h^{1-2\eta}$. Thus \eqref{LOC}, \eqref{LOC1}, \eqref{eq.Q0h}, and again the localization in $t$, yield
\begin{equation}\label{E2}
\langle \mathscr{N}_h \Xi_h^W\psi,\Xi_h^W\psi \rangle\geq \frac{\tilde{c}}{2}h^{1-2\eta}\| \Xi_h^W\psi\|^2+\mathscr{O}(h^{\infty})\| \psi \|^2\,.
\end{equation}
By using classical results of composition of pseudo-differential operators, we have
\begin{equation}
\langle \left[ \mathscr{N}_h,\Xi_h^W \right]\psi,\Xi_h^W\psi \rangle \leq Ch^{1+\eta}\| \underline\Xi_h^W \psi\|^2+\mathscr{O}(h^\infty)\|\psi\|^2\,,
    \label{E3}
\end{equation}
where $\underline \Xi$ has a support slightly larger than that of $\Xi_h$. Here we used the energy estimate
$\|\mathscr{D}_h\Xi^W_h\psi\|=\mathscr{O}(h^{1/2})\|\underline{\Xi}_h^W\psi\|+\mathscr{O}(h^\infty)\|\psi\|$, which follows from rough estimates of \eqref{E1}.

Thus, by combining \eqref{E1}, \eqref{E2}, and \eqref{E3} with the fact that $\lambda\leq Kh$, we obtain
\[
\| \Xi_h^W\psi\|^2\leq Mh^{\eta}\| \underline \Xi_h^W\psi \|^2+\mathscr{O}(h^{\infty})\| \psi \|^2\,.
\]
Finally, by an induction argument on the size of the support of $\Xi$, we get
\[\| \Xi_h^W\psi\|=\mathscr{O}(h^{\infty})\| \psi \|\,.\]
\end{proof}
Let us consider the partial semiclassical Fourier transform $\mathscr{F}_2$\footnote{which is the metaplectic transform associated with the linear symplectic application $(u_2,\upsilon_2)\mapsto (\upsilon_2,-u_2)$, see, for instance, \cite[Section 3.4]{Martinez}.} with respect to $u_2$ and the translation/dilation $T : u_1\mapsto (u_1-\upsilon_2)h^{-\frac12}=z$. With a slight abuse of notation, we identify $T$ with $\varphi\mapsto\varphi\circ T$. Letting $V=\mathscr{F}_2^{-1}T$, we have
\[V^*(-ih\partial_{u_2}-u_1)V=-h^{\frac12}z\,,\]
and, with the dilation $W : t\mapsto h^{-\frac12}t$,
\[W^*V^{*}\mathscr{D}_hVW=\hbar\mathscr{D}^\sharp_\hbar\,,\quad \hbar=h^{\frac12}\,,\]
with 
\[\mathscr{D}^\sharp_\hbar=\begin{pmatrix}
	-iC_0^\sharp\partial_{z}-tC_1^\sharp-\hbar t^2\chi(h^\eta t)^2C_2^\sharp-\hbar^2t^3\chi(h^\eta t)^3C_3^\sharp	\\
	-z-iE_0^\sharp\partial_{z}+tE_1^\sharp-\hbar t^2\chi(h^\eta t)^2{E}_2^\sharp-\hbar^2t^3\chi(h^\eta t)^3{E}_3^\sharp	\\
	-i\partial_t
\end{pmatrix}^W\]
where the coefficients of the conjugated operator $\mathscr{D}^\sharp_\hbar$ are now given by $P^\sharp=\breve P(\upsilon_2+\hbar z,-u_2)$. Here the Weyl quantization can be considered only in the variables $(u_2,\upsilon_2)$ since $z$ is now a "space variable".
We let
\[\mathscr{N}^\sharp_\hbar=[{ m_\hbar}^{-1}]^\sharp\mathscr{D}^\sharp_\hbar\cdot[{ m_\hbar}({ M_\hbar})^{-1}]^\sharp\mathscr{D}^\sharp_\hbar\,,\]
where $m_\hbar(\cdot,t)=m(\cdot,\hbar t)$ and $M_\hbar(\cdot,t)=M(\cdot,\hbar t)$. Note that $\mathscr{N}_h$ and $h\mathscr{N}^\sharp_\hbar$ are unitarily equivalent since
\begin{equation}\label{eq.Ndiese}
W^*V^{*}\mathscr{N}_hVW=h\mathscr{N}^\sharp_\hbar\,.
\end{equation}
After all these elementary transforms, Proposition \ref{prop.variable-effective} can be reformulated as follows.

\begin{proposition}\label{prop.variable-effective2}
	Let $K>0$ and $\eta\in\left(0,\frac12\right)$. Let $\Xi$ be a smooth function equal to $0$ near $0$ and $1$ away from a compact neighborhood of $0$. There exists $h_0>0$ such that for all $h\in(0,h_0)$ and for all normalized eigenfunctions $\psi$ of $\mathscr{N}^\sharp_\hbar$ associated with an eigenvalue $\lambda$ such that $\lambda\leq K$, we have
	\[\Xi\left(h^{\eta}z\right)\psi=\mathscr{O}(h^\infty)\,.\]
\end{proposition}
\begin{remark}\label{rem.roughfrequency}
As a consequence of the Agmon estimates and working in the coordinates $(u_1, u_2, t)$, we notice that the eigenfunctions are also roughly localized in "frequency" in the sense that, for all $(\alpha,\beta,\gamma)\in\mathbb{N}^3$ , and all $\eta\in\left(0,\frac12\right)$, there exist $C, h_0>0$ such that, for all $h\in(0,h_0)$,
	\[\| t^\alpha z^\beta D^\gamma_z\psi\|+\|t^\alpha z^\beta  D^\gamma_t\psi\|\leq C h^{-\eta (\alpha+\beta+\gamma)}\|\psi\|\,.\]
\end{remark}

\section{A pseudodifferential operator with operator symbol}\label{sec.4}

Proposition \ref{prop.variable-effective2} invites us to insert cutoff functions in the coefficients of the operator $\mathscr{N}^\sharp_\hbar$. That is why we consider
\begin{equation}\label{eq.Nbemol}
\mathscr{N}^{\flat}_\hbar=\left([{ m_\hbar}^{-1}]^\flat\right)^W\mathscr{D}^\flat_\hbar\cdot\left([{ m_\hbar}({ M_\hbar})^{-1}]^\flat\right)^W\mathscr{D}^\flat_\hbar\,,
\end{equation}
where
\begin{equation}\label{eq.Dbemol}
\mathscr{D}^\flat_\hbar=\begin{pmatrix}
	-iC_0^\flat\partial_{z}-tC_1^\flat-\hbar t^2\chi(h^\eta t)^2C_2^\flat-\hbar^2t^3\chi(h^\eta t)^3C_3^\flat	\\
	-z-iE_0^\flat\partial_{z}+tE_1^\flat-\hbar t^2\chi(h^\eta t)^2{E}_2^\flat-\hbar^2t^3\chi(h^\eta t)^3{E}_3^\flat	\\
	-i\partial_t
\end{pmatrix}^W\,,
\end{equation}
with $P^\flat=\breve P(\upsilon_2+\hbar\chi_\eta(z) z,-u_2)$, where $\chi_\eta(z)=\chi_0(h^{\eta}z)$, the function $\chi_0$ being smooth, with a compact support, and equal to $1$ on a neighborhood of the support of $1-\Xi$. 

\subsection{The symbol and its properties}
Expanding the operator $\mathscr{N}_\hbar^\flat$ with respect to $\hbar$ (say first at a formal level) suggests to consider the following selfadjoint operator, depending on $(u_2,\upsilon_2)$, acting as
\begin{multline*}
	n_0(u_2,\upsilon_2)\\
	=(-i\breve C_0(\upsilon_2,-u_2)\partial_{z}-t\breve C_1(\upsilon_2,-u_2))^2+\alpha^{-1}(\upsilon_2,-u_2)(-z-i\breve E_0(\upsilon_2,-u_2)\partial_z+t\breve E_1(\upsilon_2,-u_2))^2-\partial^2_{t}\,,
\end{multline*}
with the domain
\[\mathrm{Dom}(n_0)=\{\psi\in L^2(\mathbb{R}^2_+) : n_0(u_2,\upsilon_2)\psi\in L^2(\mathbb{R}^2_+)\,, \partial_t\psi(z,0)=0\}\,,\]
and where we recall that $C_1$ and $E_1$ are given in \eqref{eq.hatA}.  The domain of $n_0(u_2,\upsilon_2)$ depends on $(u_2,\upsilon_2)$. However, we can check that it is unitarily equivalent to a selfadjoint operator with domain independent of $(u_2,\upsilon_2)$, see the proof of Proposition \ref{prop.psymb} below.
In the following, we will use class of operator symbols of the form
\[S(\mathbb{R}^2,\mathcal{L}(\mathscr{A}_1,\mathscr{A}_2))=\{a\in\mathscr{C}^\infty(\mathbb{R}^2,\mathcal{L}(\mathscr{A}_1,\mathscr{A}_2)) : \forall \gamma\in\mathbb{N}^2\,, \exists C_\gamma>0 : \|\partial^\gamma a\|_{\mathcal{L}(\mathscr{A}_1,\mathscr{A}_2)}\leq C_\gamma\}\,,\]
where $\mathscr{A}_1$ and $\mathscr{A}_2$ are (fixed) Hilbert spaces. We also introduce
\[\mathscr{B}_k=\{\psi\in L^2(\mathbb{R}^2_+) :  \forall \alpha\in\mathbb{N}^k\,,|\alpha|\leq k\Rightarrow(\langle t\rangle^{k}+\langle z\rangle^{k})\partial^\alpha\psi \in L^2(\mathbb{R}^2_+)\}\,,\]
and the class of symbols
\[S(\mathbb{R}^2,N)=\bigcap_{k\geq N}S(\mathbb{R}^2,\mathcal{L}(\mathscr{B}_{k},\mathscr{B}_{k-N}))\,.\]
and we notice that $n_0\in S(\mathbb{R}^2,2)$.  
\begin{remark}\label{rem.SMN}
Note  that these classes of symbols are not algebras. However, the classical Moyal product of symbols in $S(\mathbb{R}^2,N)$ and $S(\mathbb{R}^2,M)$ is well-defined and belongs to $S(\mathbb{R}^2,N+M)$, see \cite[Theorem 2. 1. 12]{Keraval}.
\end{remark}

In fact, for $N\geq 2$, by using a classical trace theorem, we may also define
\[\mathscr{B}^{\mathrm{Neu}}_N=\{\psi\in \mathscr{B}_N : \partial_t\psi(z,0)=0 \}(\subset \mathrm{Dom}\, n_0)\,,\]
and the associated class $S^{\mathrm{Neu}}(\mathbb{R}^2,N)$.
We can also write  $n_0\in S^{\mathrm{Neu}}(\mathbb{R}^2,2)$ to remember that the domain of $n_0$ is equipped with the Neumann condition.

By expanding $\mathscr{N}_\hbar^\flat$ in powers of $\hbar$ and by using a composition theorem for pseudodifferential operators, we get the following. 
\begin{proposition}\label{prop.Nhbarbemol}
The operator $\mathscr{N}_\hbar^\flat$ is an $h$-pseudodifferential operator with symbol in the class $S^{\mathrm{Neu}}(\mathbb{R}^2,2)$.
Moreover, we can write the expansion
\begin{equation}\label{eq.expNbemol}
\mathscr{N}_\hbar^\flat=n_0^W+\hbar n_1^W+\hbar^2n_2^W+\hbar^3r^W_\hbar\,,
\end{equation}
with $n_1$, $n_2$ and $r_\hbar$ in the class $S^{\mathrm{Neu}}(\mathbb{R}^2,8)$. 
\end{proposition}
\begin{proof}
Let us recall that $\mathscr{N}_\hbar^\flat$ is given in \eqref{eq.Nbemol}. Let us notice that the operator $\mathscr{D}^\flat_\hbar$, defined in \eqref{eq.Dbemol}, is indeed a pseudodifferential operator with operator-valued symbol. With respect to the variables $z$ and $t$, it is a differential operator of order $1$ whose symbol is
\begin{equation}\label{eq.symbDbemol}
\begin{pmatrix}
	-iC_0^\flat\partial_{z}-tC_1^\flat-\hbar t^2\chi(h^\eta t)^2C_2^\flat-\hbar^2t^3\chi(h^\eta t)^3C_3^\flat	\\
	-z-iE_0^\flat\partial_{z}+tE_1^\flat-\hbar t^2\chi(h^\eta t)^2{E}_2^\flat-\hbar^2t^3\chi(h^\eta t)^3{E}_3^\flat	\\
	-i\partial_t
\end{pmatrix}
\end{equation}
and belongs to $S(\mathbb{R}^2,1)$. The functions/symbols $[{ m_\hbar}^{-1}]^\flat$ and $[{ m_\hbar}({ M_\hbar})^{-1}]^\flat$ belong to $S(\mathbb{R}^2,0)$. Combining these considerations with \eqref{eq.Nbemol}, it remains to apply the composition theorem for pseudodifferential operators with operator symbols, see Remark \ref{rem.SMN}.

To get \eqref{eq.expNbemol}, it is sufficient to use the Taylor expansions in $\hbar$ of the symbol \eqref{eq.symbDbemol}, $[{ m_\hbar}^{-1}]^\flat$, and $[{ m_\hbar}({ M_\hbar})^{-1}]^\flat$, and to apply again the composition theorem (the worst remainders being roughly of order $8$ in $(z,t)$).
\end{proof}
\begin{remark}
We will see that the accurate description of $n_1$ and $n_2$ in \eqref{eq.expNbemol} are not necessary to prove our main theorem. The use of the more restrictive class $S^{\mathrm{Neu}}(\mathbb{R}^2,8)$ allows to deal with the uniformity in the semiclassical expansions in $\hbar$.
\end{remark}
Let us describe the groundstate energy of the principal symbol $n_0$. From now on, we lighten the notation by setting $(u_2,\upsilon_2)=(u,\upsilon)$. 

\begin{proposition}\label{prop.psymb}
For all $(u,\upsilon)\in\mathbb{R}^2$, the bottom of the spectrum of $n_0$ belongs to the discrete spectrum and it is a simple eigenvalue that equals $\breve\beta(\upsilon,-u)$.	The corresponding normalized eigenfunction $\mathfrak{f}_{u,\upsilon}$ belongs to the Schwartz class and depends on $(u,\upsilon)$ smoothly.

Moreover, there exists $c>0$ such that, by possibly choosing $\epsilon_0$ smaller in Lemma \ref{lem.epsilon0}, we have, for all $(u,\upsilon)\in\mathbb{R}^2$,
\[\inf\mathrm{sp}(n_0(u,\upsilon)|_{\mathfrak{f}^\perp_{u,\upsilon}})\geq \beta_{\min}+c\geq \breve\beta(u,\upsilon)\,.\]

\end{proposition}
\begin{proof}
By using the Fourier transform in $z$ and then a change of gauge, we are reduced to the case when $E_0=0$. With a rescaling in $z$, $n_0$ is unitarily equivalent to
\begin{equation*}
(-i\partial_z- t\breve C_1)^2+\alpha^{-1}(-\breve C_0 z+t\breve E_1)^2-\partial^2_{t}=(-i\partial_{z}- tb_2)^2+(b_3 z+tb_1)^2-\partial^2_{ t}\,,
\end{equation*}
with
\[b_1=\breve{\mathcal{B}}_1\,,\quad b_2=\breve{\alpha^{\frac12}\mathcal{B}}_2\,,\quad b_3=-\breve{\mathcal{B}}_3\,,\]
where the functions are evaluated at $(\upsilon_2,-u_2)$. Recalling \eqref{eq.newB}, we see that the Euclidean norm of $b=(b_1, b_2, b_3)$ is
\[\|b\|_2=\|\breve{\mathbf{B}}\|\,,\]
with a slight abuse of notation. By homogeneity, we can easily scale out $\|\breve{\mathbf{B}}\|$ and consider the operator
\[(-i\partial_z-tb_2)^2-\partial_t^2+(tb_1+b_3z)^2\,,\]
with $b_1=\cos\theta\cos\varphi$, $b_2=\cos\theta\sin\varphi$ and $b_3=\sin\theta$. Completing a square leads to the identity
\begin{multline*}
(-i\partial_z-tb_2)^2-\partial_t^2+(tb_1+b_3z)^2\\
=-\partial_t^2+(t\cos\theta-\sin\varphi D_z-z\sin\theta\cos\varphi)^2+(\cos\varphi D_z-z\sin\theta\sin\varphi)^2\,.
\end{multline*}
This shows, thanks to a change of gauge and a rescaling in $z$, that the operator is unitarily equivalent to
\[D_t^2+(t\cos\theta-\tan\varphi D_z-z\sin\theta )^2+D^2_z\]
and then, by a change of gauge on the Fourier side, to
\[D_t^2+D^2_z+(t\cos\theta-z\sin\theta )^2\,,\]
which is nothing but the Lu-Pan operator defined in \eqref{eq.LuPan}, which is unitarily equivalent to $\cos^2\theta D_t^2+\sin^2\theta D^2_z+(t-z)^2$  (whose domain is independent of $\theta$).

The eigenfunction $\mathfrak{f}_{u,\upsilon}$ belongs to the Schwartz class in virtue of \cite[Corollaire 5.1.2]{Raymond09} and the stability of the Schwartz class by Fourier and gauge transforms.

\end{proof}

\subsection{An approximate parametrix}

\subsubsection{Inverting the principal symbol}

\begin{lemma}\label{lem.Grushin0}
Consider $\epsilon>0$ and $\Lambda\leq \beta_{\min}+\epsilon$. We let
\[\mathscr{P}_0(\Lambda)=\begin{pmatrix}
n_0(u,\upsilon)-\Lambda&\cdot \mathfrak{f}_{u,\upsilon}\\
\langle\cdot,\mathfrak{f}_{u,\upsilon}\rangle&0
	\end{pmatrix}\,.
	\]	
	For $\epsilon$ small enough, $\mathscr{P}_0(\Lambda) : \mathrm{Dom}\, n_0\times\mathbb{C}\to L^2(\mathbb{R}^2_+)\times\mathbb{C}$ is bijective. Its inverse is denoted by $\mathscr{Q}_0$ and is given by
	\[\mathscr{Q}_0=	\mathscr{Q}_0(\Lambda)=\begin{pmatrix}
	(n_0(u,\upsilon)-\Lambda)^{-1}_\perp&\cdot\mathfrak{f}_{u,\upsilon}\\
		\langle\cdot, \mathfrak{f}_{u,\upsilon}\rangle&\Lambda-\breve\beta(\upsilon,u)
	\end{pmatrix}\,,\]
where $(n_0(u,\upsilon)-\Lambda)^{-1}_\perp$ is the regularized resolvent on $(\mathrm{span}\,\mathfrak{f}_{u,\upsilon})^\perp$.

Moreover, we have $\mathscr{Q}_0\in S(\mathbb{R}^2,0)$.
\end{lemma}

\begin{proof}
	By using the same algebraic computations as in \cite{Keraval} and the spectral gap in Proposition \ref{prop.psymb}, we get the announced inverse. Moreover, it is also clear that $\mathscr{Q}_0$ is bounded from $L^2(\R^2_+)$ to $L^2(\R^2_+)$ uniformly in $(u,\upsilon)$. The fact that it belongs to the class $S(\R^2,0)$ follows from weighted resolvent estimates similar to \cite[p.100-101]{Raymond09}, see also \cite[Appendix]{FLTRVN22}.
	\end{proof} 

We let
\[\mathscr{P}_\hbar(\Lambda)=\begin{pmatrix}
	n_0+\hbar n_1+\hbar^2n_2+\hbar^3r_\hbar-\Lambda&	\cdot \mathfrak{f}_{u,\upsilon}\\
\langle\cdot,\mathfrak{f}_{u,\upsilon}\rangle&0
\end{pmatrix}=\mathscr{P}_0(\Lambda)+\hbar\mathscr{P}_1+\hbar^2\mathscr{P}_2+\hbar^3\mathscr{R}_\hbar\,,
\]	
where $n_0$, $n_1$, $n_2$, and $r_\hbar$ are given in Proposition \ref{prop.Nhbarbemol}.
\subsubsection{The approximate parametrix}
Let us now construct an approximate (at the order $2$) inverse of $\mathscr{P}_\hbar^W$ when it acts on the Schwartz class (with Neumann condition). We consider
\[\mathscr{Q}_\hbar=\mathscr{Q}_0+\hbar\mathscr{Q}_1+\hbar^2\mathscr{Q}_2=\begin{pmatrix}
Q_\hbar&Q_\hbar^+\\
Q_\hbar^-&Q_{\hbar}^\pm
\end{pmatrix}\,,\]
where
\begin{equation}\label{eq.Q1Q2}
\mathscr{Q}_1=-\mathscr{Q}_0\mathscr{P}_1\mathscr{Q}_0\,,\quad \mathscr{Q}_2=-\mathscr{Q}_0\mathscr{P}_2\mathscr{Q}_0+\mathscr{Q}_0\mathscr{P}_1\mathscr{Q}_0\mathscr{P}_1\mathscr{Q}_0-\frac1i\{\mathscr{Q}_0,\mathscr{P}_0\}\mathscr{Q}_0\,.
\end{equation}
By Remark \ref{rem.SMN}, the symbols $\mathscr{Q}_1$ and $\mathscr{Q}_2$ belong to $S(\mathbb{R}^2,M)$, for some $M\geq 8$. By computing products of matrices and using the exponential decay of $\mathfrak{f}_{u,\upsilon}$, we get
\begin{equation}\label{eq.Qpm}
Q_\hbar^\pm(\Lambda)=\Lambda-(p_0+\hbar p_1+\hbar^2p_{2,\Lambda})\,,
\end{equation}
with $p_0=\breve\beta(\upsilon,-u)$ and $p_1, p_{2,\Lambda}\in S_{\mathbb{R}^2}(1)$ where
\[S_{\mathbb{R}^2}(1)=\{a\in\mathscr{C}^\infty(\mathbb{R}^2,\mathbb{C}) : \forall\alpha\in\mathbb{N}^2\,, \exists C_\alpha>0 : |\partial^\alpha a|\leq C_\alpha\}\,.\]
In addition, $\Lambda\mapsto p_{2,\Lambda}\in S_{\mathbb{R}^2}(1)$ is analytic in a neighborhood of $\beta_{\min}$.
\begin{remark}
Let us emphasize here that nothing a priori ensures that the subprincipal symbols $p_1$ and $p_{2,E}$ are real-valued since our formal operator is not selfadjoint on the canonical $L^2$-space.	
\end{remark}
The reason to consider the expressions \eqref{eq.Q1Q2} simply comes from the semiclassical expansion of the product $\mathscr{Q}_\hbar^W\mathscr{P}_\hbar^W$ by means of the composition theorem \cite[Theorem 2.1.12]{Keraval}. These explicit choices, with the Calder\'on-Vaillancourt Theorem \cite[Theorem 2.1.16]{Keraval} to estimate the remainders, imply the following proposition.

\begin{proposition}\label{prop.parametrix}
There exists $N\geq 2$ such that the following holds. We have
\[\mathscr{Q}_\hbar^W\mathscr{P}_\hbar^W=\mathrm{Id}_{\mathscr{S}^\mathrm{Neu}(\overline{\mathbb{R}}^2_+)\times\mathscr{S}(\R)}+\hbar^3\mathscr{R}^W_{\hbar,\ell}\,,\quad \mathscr{P}_\hbar^W\mathscr{Q}_\hbar^W=\mathrm{Id}_{\mathscr{S}(\overline{\mathbb{R}}^2_+)\times\mathscr{S}(\R)}+\hbar^3\mathscr{R}^W_{\hbar,r}\,,\]
where $\mathscr{R}_{\hbar,\ell}$ and $\mathscr{R}_{\hbar,r}$ belong to $S(\mathbb{R}^2,N)$ and where $\mathscr{S}^\mathrm{Neu}(\overline{\mathbb{R}}^2_+)$ denotes the Schwartz class on $\mathbb{R}^2_+$ with Neumann condition at $t=0$.

In particular, we have, for all $\psi\in\mathscr{S}^\mathrm{Neu}(\overline{\mathbb{R}}^2_+)$,
\begin{equation}\label{eq.Grul}
	\begin{split}
	Q^W_\hbar(\mathscr{N}^\flat_\hbar-\Lambda)\psi+(Q_\hbar^+)^W\mathfrak{P}\psi&=\psi+\mathscr{O}(\hbar^3)\|\psi\|_{ L^2(\mathbb{R},\mathscr{B}_N)}\,,\\
	(Q^-_\hbar)^W(\mathscr{N}^\flat_\hbar-\Lambda)\psi+(Q_\hbar^\pm)^W\mathfrak{P}\psi	&=\mathscr{O}(\hbar^3)\|\psi\|_{ L^2(\mathbb{R},\mathscr{B}_N)}\,,
\end{split}\end{equation}
and, for all $\varphi\in\mathscr{S}(\mathbb{R})$,
\begin{equation}\label{eq.Grur}
	\begin{split}
	(\mathscr{N}^\flat_\hbar-\Lambda)(Q_\hbar^+)^W\varphi+\mathfrak{P}^*(Q_\hbar^\pm)^W\varphi&=\mathscr{O}(\hbar^3)\|\varphi\|\,,\\
\mathfrak{P}(Q_\hbar^+)^W\varphi	&=\varphi+\mathscr{O}(\hbar^3)\|\varphi\|\,.
\end{split}
\end{equation}
Here, $\mathfrak{P}=(\langle\cdot,\mathfrak{f}_{u,\upsilon}\rangle)^W$.
\end{proposition}

\section{Spectral consequences}\label{sec.5}
This last section is devoted to the proof of Theorem \ref{thm.main} with the help of Proposition \ref{prop.parametrix}. The spectrum of $\mathscr{N}^\sharp_\hbar$ will be compared to the spectrum of a model operator, derived from an effective operator whose symbol is
\begin{equation}\label{eq.peff}
p_h^{\mathrm{eff}}=p_0+\hbar p_1+\hbar^2p_{2,\beta_{\min}}\,,
\end{equation}
see \eqref{eq.Qpm}.
\subsection{A model operator}\label{sec.model}
Let us consider
\[p_h^{\mathrm{mod}}(U)=p^{\mathrm{eff}}_h(0)+\frac12\mathrm{Hess}_{(0,0)}\, p_0 (U,U)+\hbar p^{\mathrm{lin}}_1(U)\,,\quad U=(u,\upsilon)\,,\]
where $p^{\mathrm{lin}}_1$ is the linear approximation of $p_1$ at $(0,0)$. The corresponding operator $(p_h^{\mathrm{mod}})^W$ is not selfadjoint due to the linear part. However, this operator has still compact resolvent and we can compute its spectrum and estimate its resolvent. Let us explain this. Thanks to a rotation and Assumption \ref{hyp.main}, we may assume that
\[p_h^{\mathrm{mod}}=p^{\mathrm{eff}}_h(0)+\frac{d_0}{2}(u^2+\upsilon^2)+\hbar(\alpha u+\beta\upsilon)\,,\]
for some $d_0>0$ and $(\alpha,\beta)\in\mathbb{C}^2$. 

\begin{remark}\label{rem.d0}
In fact, we have
\[d_0=\sqrt{\det\mathrm{Hess}_{(0,0)} p_0}=\sqrt{\det\mathrm{Hess}_{(0,0)} \breve\beta(\upsilon,-u)}=\sqrt{\frac{\det\mathrm{Hess}_{x_0} \beta}{{\|\mathbf{B}(x_0)\|^2\sin^2\theta(x_0)}}}\,,\]
where we used the notation introduced at the beginning of Section \ref{sec.3}, the change of variable $\mathscr{J}$ in Lemma \ref{lem.J}, and Remark \ref{rem.B3}.
\end{remark}

By completing  square, we get
\[(p_h^{\mathrm{mod}})^W=\tilde p^{\mathrm{eff}}_h(0)+\frac{d_0}{2}\left(\left(u+\frac{\hbar\alpha}{d_0}\right)^2+\left(hD_u+\frac{\hbar\beta}{d_0}\right)^2\right)\,\,, \quad \tilde p^{\mathrm{eff}}_h(0)=p^{\mathrm{eff}}_h(0)-\frac{\alpha^2+\beta^2}{d_0}h\,.\]
For all $n\geq 1$, we let
\[f_{n}(u)=[e^{-i\beta \cdot/d_0}H_n(\cdot)]\left(u+\frac{\alpha}{d_0}\right)\,,\quad  f_{n,\hbar}(u)=\hbar^{-\frac12}f_n(\hbar^{-1}u)\,,\]
where $H_n$ is the $n$-th normalized Hermite function. 

The family $(f_{n,\hbar})_{n\geq 1}$ is a total family in $L^2(\mathbb{R})$ (but not necessarily orthogonal). It satisfies
\begin{equation}\label{eq.lambdanmod}
(p_h^{\mathrm{mod}})^Wf_{n,\hbar}=\lambda_n^{\mathrm{mod}}(h)f_{n,\hbar}\,,\quad  \lambda_n^{\mathrm{mod}}(h)=\frac{d_0}{2}(2n-1)h+\tilde p^{\mathrm{eff}}_h(0)\,.
\end{equation}
By the analytic perturbation theory, the spectrum of $(p_h^{\mathrm{mod}})^W$ is made of eigenvalues of algebraic multiplicity $1$ and it is given by
\[\mathrm{sp}\left((p_h^{\mathrm{mod}})^W\right)=\left\{\frac{d_0}{2}(2n-1)h+\tilde p^{\mathrm{eff}}_h(0)\,,\quad n\geq 1\right\}\,.\]
Moreover, for all compact $K\subset\mathbb{C}$, there exists $C_K>0$ such that, for all $\mu\in K$,
\begin{equation}\label{eq.resolvent}
\|((p_h^{\mathrm{mod}})^W-\tilde p^{\mathrm{eff}}_h(0)-h\mu)^{-1}\|\leq \frac{C_K}{\mathrm{dist}(\tilde p^{\mathrm{eff}}_h(0)+h\mu,\mathrm{sp}\left((p_h^{\mathrm{mod}})^W\right))}\,.
\end{equation}

\subsection{Refined estimates}

\subsubsection{From the model operator to $\mathscr{N}_\hbar^\sharp$}\label{sec.quasimodes}
The functions $(f_{n,\hbar})$ can serve as quasimodes for $\mathscr{N_\hbar}^\sharp$ with the help of \eqref{eq.Grur}. Indeed, by taking $z=\lambda_n^{\mathrm{mod}}(h)$ and $\varphi=f_{n,\hbar}$, we see that
\[	(\mathscr{N}^\flat_\hbar-\lambda_n^{\mathrm{mod}}(h))(Q_\hbar^+)^Wf_{n,\hbar}=\mathscr{O}(\hbar^3)\,,\]
Since $(Q_\hbar^+)^Wf_{n,\hbar}$ is  localized near $(z,t)=(0,0)$ (due to the exponential decay of $\mathfrak{f}_{u,\upsilon}$, which is uniform in $(u,\upsilon)$), we get
\[	(\mathscr{N}^\sharp_\hbar-\lambda_n^{\mathrm{mod}}(h))(Q_\hbar^+)^Wf_{n,\hbar}=\mathscr{O}(\hbar^3)\,.\]
By using the inverse Fourier transform and translation/dilation, $(Q_\hbar^+)^Wf_{n,\hbar}$ becomes a quasimode for $\mathscr{N}_h$, see \eqref{eq.LapptoNh} and the end of Section \ref{sec.3}. But the operator $\mathscr{N}_h$ is unitarily equivalent to selfadjoint for a suitable scalar product on the usual $L^2$-space. Therefore, we can apply the spectral theorem and we deduce that
\[\mathrm{dist}\left(\lambda_n^{\mathrm{mod}}(h),\mathrm{sp}(\mathscr{N}^\sharp_\hbar)\right)\leq C\hbar^3\,.\]
In particular, this implies that, for $h$ small enough, $\lambda_n^{\mathrm{mod}}(h)$ is real. This shows that we necessarily have
\[p_1(0)\in\mathbb{R}\,,\quad p_2(0)-\frac{\alpha^2+\beta^2}{d_0}\in\mathbb{R}\,.\]
This also implies that
\begin{equation}\label{eq.upperbound}
\lambda_n(\mathscr{N}_\hbar^\sharp)\leq \lambda_n^{\mathrm{mod}}(h)+C\hbar^3\,.
\end{equation}

\subsubsection{From $\mathscr{N}_\hbar^\sharp$ to the model operator}\label{sec.lowerbound}
Let $n\geq 1$. Let us consider an eigenfunction $\psi$ of $\mathscr{N}_\hbar^\sharp$ associated with the eigenvalue $\lambda_n(\mathscr{N}_\hbar^\sharp)$. 

We know that $\lambda_n(\mathscr{N}_\hbar^\sharp)=\beta_{\min}+o(1)$ and that the corresponding eigenfunctions are localized in $(z,t)$ (due to the Agmon estimates and Proposition \ref{prop.variable-effective2}). Thus, in \eqref{eq.Grul}, we can replace $\mathscr{N}_\hbar^\flat$ by  $\mathscr{N}_\hbar^\sharp$ and we deduce that
\begin{equation}\label{eq.quasimodepeff0}
\left((p_h^{\mathrm{eff}})^W-\lambda_n(\mathscr{N}_\hbar^\sharp)\right)\mathfrak{P}\psi=\mathscr{O}(\hbar^3)\|\psi\|\,,\quad \|\psi\|\leq C\|\mathfrak{P}\psi\|\,.
\end{equation}
where we used Remark \ref{rem.roughfrequency} to control the remainders.
By taking the scalar product with $\mathfrak{P}\psi$, taking the real part and using the min-max principle, we get that
\[\lambda_n(\mathscr{N}_\hbar^\sharp)\geq \beta_{\min}+p_1(0)\hbar-Ch\,.\]
This establishes the two-term asymptotic estimate
\[\lambda_n(\mathscr{N}_\hbar^\sharp)=\beta_{\min}+p_1(0)\hbar+\mathscr{O}(h)\,.\]
Therefore, we can focus on the description of the eigenvalues of the form
\[\lambda_n(\mathscr{N}_\hbar^\sharp)=\beta_{\min}+p_1(0)\hbar+\mu_n(\hbar) h\,,\]
for $\mu_n(\hbar)\in D(0,R)$ with a given $R>0$. We have
\begin{equation}\label{eq.quasimodepeff}
\left((p_h^{\mathrm{eff}})^W-(\beta_{\min}+p_1(0)\hbar+\mu_n(\hbar) h))\right)\mathfrak{P}\psi_n=\mathscr{O}(\hbar^3)\|\mathfrak{P}\psi_n\|\,,
\end{equation}
where $\psi_n$ denotes a normalized eigenfunction associated to the $n$-th eigenvalue of $\mathscr{N}_\hbar^\sharp$. In fact, by considering \eqref{eq.quasimodepeff} and again Proposition \ref{prop.parametrix}, the function $\mathfrak{P}\psi_n$ is microlocalized near $(0,0)$, the minimum of the principal symbol $p_0$. Since this minimum is non-degenerate, the quadratic approximation of the symbol shows that $\mathfrak{P}\psi_n$ is microlocalized near $(u,\upsilon)=(0,0)$ at the scale $\hbar^{1-\eta}$ for any $\eta\in\left(0,\frac12\right)$. In particular, we deduce that
\[\left((p_h^{\mathrm{mod}})^W-(\beta_{\min}+p_1(0)\hbar+\mu_n(\hbar) h))\right)\mathfrak{P}\psi_n=\mathscr{O}(\hbar^{3-3\eta})\|\mathfrak{P}\psi_n\|\,.\]
From the resolvent estimate \eqref{eq.resolvent}, this implies that
\[\mu_n(\hbar)\in \bigcup_{j\geq 1} D\left(\frac{d_0}{2}(2j-1)+d_1, C\hbar^{1-3\eta}\right)\,,\quad d_1=p_2(0)-\frac{\alpha^2+\beta^2}{d_0}\,,\]
where $D(z,r)$ denotes the disc of center $z\in\C$ and radius $r>0$.
In particular, we have
\[\mu_1(\hbar)\geq \frac{d_0}{2}+d_1-C\hbar^{1-3\eta}\,.\]
This shows that
\[\lambda_1\left(\mathscr{N}_\hbar^\sharp\right)\geq \beta_{\min}+p_1(0)\hbar+\left(\frac{d_0}{2}+d_1\right)\hbar^2-C\hbar^{3-3\eta}\,,\]
and thus, with \eqref{eq.upperbound}, we get
\[\mu_1(\hbar)= \frac{d_0}{2}+d_1+\mathscr{O}(\hbar^{1-3\eta})\,,\]
and
\[\lambda_1\left(\mathscr{N}_\hbar^\sharp\right)= \lambda_1^{\mathrm{mod}}(h)+\mathscr{O}(\hbar^{3-3\eta})\,.\]
Let us now deal with $\lambda_2\left(\mathscr{N}_\hbar^\sharp\right)$ and recall \eqref{eq.upperbound}. Assume by contradiction that $\mu_2(\hbar)\in D\left(\frac{d_0}{2}+d_1, C\hbar^{1-3\eta}\right)$. Then, we have
\[|\mu_2(\hbar)-\mu_1(\hbar)|\leq C\hbar^{1-3\eta}\,.\]
We infer that
\[\left((p_h^{\mathrm{mod}})^W-\lambda_1^{\mathrm{mod}}(h)\right)\mathfrak{P}\psi=\mathscr{O}(\hbar^{3-3\eta})\|\mathfrak{P}\psi\|\,,\]
for all $\psi\in\mathrm{span}(\psi_1,\psi_2)$. Moreover, coming back to \eqref{eq.Grul} (see also \eqref{eq.quasimodepeff}), we also get that $\|\psi\|\leq C\|\mathfrak{P}\psi\|$ for all $\psi\in\mathrm{span}(\psi_1,\psi_2)$. In particular, $\mathfrak{P}(\mathrm{span}(\psi_1,\psi_2))$ is of dimension two. Let us consider the Riesz projector (in the characteristic subspace of $(p_\hbar^{\mathrm{mod}})^W$ associated with the smallest eigenvalue)
\[\Pi=\frac{1}{2i\pi}\int_{\lambda_1^{\mathrm{mod}}(h),\hbar^{3-4\eta})}(\zeta-(p_\hbar^{\mathrm{mod}})^W)^{-1}\mathrm{d}\zeta\,,\]
which is of rank one. Then, for all $\varphi\in\mathfrak{P}(\mathrm{span}(\psi_1,\psi_2))$, we write, with the Cauchy formula,
\begin{equation*}
\Pi\varphi=\varphi
+\frac{1}{2i\pi}\int_{\lambda_1^{\mathrm{mod}}(h),\hbar^{3-4\eta})}\left((\zeta-(p_\hbar^{\mathrm{mod}})^W)^{-1}-(\zeta-\lambda_1^{\mathrm{mod}}(h))^{-1}\right)\mathrm{d}\zeta\,.
\end{equation*}
But, we have
\[(\zeta-(p_\hbar^{\mathrm{mod}})^W)^{-1}-(\zeta-\lambda_1^{\mathrm{mod}}(h))^{-1}=(\zeta-\lambda_1^{\mathrm{mod}}(h))^{-1}(\zeta-(p_\hbar^{\mathrm{mod}})^W)^{-1}\left((p_h^{\mathrm{mod}})^W-\lambda_1^{\mathrm{mod}}(h))\right)\,,\]
so that, by using the resolvent estimate \eqref{eq.resolvent}, we get
\[\|\Pi\varphi-\varphi\|\leq C\hbar^{3-4\eta}\hbar^{-3+4\eta}\hbar^{-3+4\eta}\hbar^{3-3\eta}\|\varphi\|=C\hbar^{\eta}\|\varphi\|\,.\]
This shows that the range of $\Pi$ is of dimension at least two as soon as $\hbar$ is small enough. This is a contradiction. Therefore, we must have $\mu_2(\hbar)\in D\left(3\frac{d_0}{2}+d_1, C\hbar^{1-3\eta}\right)$. In particular, we have
\[\mu_2(\hbar)= 3\frac{d_0}{2}+d_1+\mathscr{O}(\hbar^{1-3\eta})\,,\quad \lambda_2\left(\mathscr{N}_\hbar^\sharp\right)= \lambda_2^{\mathrm{mod}}(h)+\mathscr{O}(\hbar^{3-3\eta})\,. \]
We proceed by induction to get that, for all $n\geq 1$,
\begin{equation}\label{eq.lambdandiese}
\mu_n(\hbar)= (2n-1)\frac{d_0}{2}+d_1+\mathscr{O}(\hbar^{1-3\eta})\,,\quad \lambda_n\left(\mathscr{N}_\hbar^\sharp\right)= \lambda_n^{\mathrm{mod}}(h)+\mathscr{O}(\hbar^{3-3\eta})\,. 
\end{equation}

\subsubsection{End of the proof of Theorem \ref{thm.main}}
Proposition \ref{prop.Lapp} shows that  the first eigenvalues of $\mathscr{L}_h$ coincide with those of $\mathscr{L}^{\mathrm{app}}_h$ modulo $o(h^2)$. Then, by \eqref{eq.LapptoNh}, $\mathscr{L}^{\mathrm{app}}_h$ is unitarily equivalent to $\mathscr{N}_h$. The operator $\mathscr{N}_h$ is unitarily equivalent to $h\mathscr{N}_\hbar^\sharp$, see \eqref{eq.Ndiese}. Theorem \ref{thm.main} follows from \eqref{eq.lambdandiese} and \eqref{eq.lambdanmod} (see also Remark \ref{rem.d0} for the explicit formula for $d_0$).

\section*{Acknowledgments}
This work was conducted within the France 2030 framework programme, Centre Henri Lebesgue ANR-11-LABX-0020-01.

\bibliographystyle{abbrv}
\bibliography{bibaaahr}

\begin{thebibliography}{10}

\bibitem{BHR21}
V.~Bonnaillie-No\"{e}l, F.~H\'{e}rau, and N.~Raymond.
\newblock Purely magnetic tunneling effect in two dimensions.
\newblock {\em Invent. Math.}, 227(2):745--793, 2022.

\bibitem{FLTRVN22}
R.~Fahs, L.~Le~Treust, N.~Raymond, and S.~V\~{u}~Ng\d{o}c.
\newblock {Edge states for the Robin magnetic Laplacian}.
\newblock 2022.

\bibitem{FH10}
S.~Fournais and B.~Helffer.
\newblock {\em Spectral methods in surface superconductivity}, volume~77 of
  {\em Progress in Nonlinear Differential Equations and their Applications}.
\newblock Birkh\"{a}user Boston, Inc., Boston, MA, 2010.

\bibitem{FHKR22}
S.~Fournais, B.~Helffer, A.~Kachmar, and N.~Raymond.
\newblock Effective operators on an attractive magnetic edge.
\newblock {\em J. {\'E}c. Polytech., Math.}, 10:917--944, 2023.

\bibitem{HK22}
B.~Helffer and A.~Kachmar.
\newblock Helical magnetic fields and semi-classical asymptotics of the lowest
  eigenvalue.
\newblock {\em https://arxiv.org/abs/2204.09880}.

\bibitem{HKRVN16}
B.~Helffer, Y.~Kordyukov, N.~Raymond, and S.~V\~{u}~Ng\d{o}c.
\newblock Magnetic wells in dimension three.
\newblock {\em Anal. PDE}, 9(7):1575--1608, 2016.

\bibitem{HM02}
B.~Helffer and A.~Morame.
\newblock Magnetic bottles for the {N}eumann problem: the case of dimension 3.
\newblock volume 112, pages 71--84. 2002.
\newblock Spectral and inverse spectral theory (Goa, 2000).

\bibitem{HM04}
B.~Helffer and A.~Morame.
\newblock Magnetic bottles for the {N}eumann problem: curvature effects in the
  case of dimension 3 (general case).
\newblock {\em Ann. Sci. \'{E}cole Norm. Sup. (4)}, 37(1):105--170, 2004.

\bibitem{HR22}
F.~Hérau and N.~Raymond.
\newblock {Semiclassical spectral gaps for the 3D Neumann Laplacian with
  constant magnetic field}.
\newblock {\em To appear in Annales de l'Institut Fourier}.

\bibitem{Keraval}
P.~Keraval.
\newblock {\em Formules de Weyl par réduction de dimension. Applications à
  des Laplaciens électro-magnétiques}.
\newblock PhD thesis, Université de Rennes 1, 2018.

\bibitem{LP00}
K.~Lu and X.-B. Pan.
\newblock Surface nucleation of superconductivity in 3-dimensions.
\newblock volume 168, pages 386--452. 2000.
\newblock Special issue in celebration of Jack K. Hale's 70th birthday, Part 2
  (Atlanta, GA/Lisbon, 1998).

\bibitem{Martinez}
A.~Martinez.
\newblock {\em An introduction to semiclassical and microlocal analysis}.
\newblock Universitext. New York, NY: Springer, 2002.

\bibitem{MRVN22}
L.~Morin, N.~Raymond, and S.~V\~{u}~Ng\d{o}c.
\newblock {Eigenvalue asymptotics for confining magnetic Schrödinger operators
  with complex potentials }.
\newblock {\em To appear in Int. Math. Res. Not.}

\bibitem{Raymond09}
N.~Raymond.
\newblock {\em Méthodes spectrales et théorie des cristaux liquides}.
\newblock PhD thesis, Université Paris-Sud 11, 2009.

\bibitem{R10}
N.~Raymond.
\newblock On the semiclassical 3{D} {N}eumann {L}aplacian with variable
  magnetic field.
\newblock {\em Asymptot. Anal.}, 68(1-2):1--40, 2010.

\bibitem{R10b}
N.~Raymond.
\newblock Uniform spectral estimates for families of {S}chr\"{o}dinger
  operators with magnetic field of constant intensity and applications.
\newblock {\em Cubo}, 12(1):67--81, 2010.

\bibitem{R12}
N.~Raymond.
\newblock Semiclassical 3{D} {N}eumann {L}aplacian with variable magnetic
  field: a toy model.
\newblock {\em Comm. Partial Differential Equations}, 37(9):1528--1552, 2012.

\bibitem{Raymond17}
N.~Raymond.
\newblock {\em Bound states of the magnetic {S}chr\"{o}dinger operator},
  volume~27 of {\em EMS Tracts in Mathematics}.
\newblock European Mathematical Society (EMS), Z\"{u}rich, 2017.

\bibitem{SZ07}
J.~Sj\"{o}strand and M.~Zworski.
\newblock Elementary linear algebra for advanced spectral problems.
\newblock volume~57, pages 2095--2141. 2007.
\newblock Festival Yves Colin de Verdi\`ere.

\end{thebibliography}

	\end{document}